\documentclass[11pt]{article}
\usepackage{amsfonts,amsmath,amsthm}
\usepackage{graphicx}
\usepackage{amsfonts,amssymb,widebar}

\usepackage[margin=0.97in,top=.99in,bottom=1.1in, textheight=22.51cm,textwidth=16.51cm, heightrounded,columnsep=0.81cm]{geometry}

  \title{{Ordinary Differential Equation Methods
\\ 
For Markov Decision Processes
 \\
 and Application to 
Kullback--Leibler Control Cost}\thanks{Research   supported by
French National Research Agency grant ANR-12-MONU-0019,  and  NSF grants 1609131 and CPS-1259040.}}

\author{Ana Bu\v{s}i\'{c}\thanks{Inria and the Computer Science Department of \'Ecole Normale Sup\'erieure, Paris, France
    ({ana.busic@inria.fr}, \url{http://www.di.ens.fr/\string~busic/}).}
  \and
  Sean Meyn\thanks{Department of Electrical and Computer
Engineering at the University of Florida, Gainesville ({meyn@ece.ufl.edu}, \url{http://www.meyn.ece.ufl.edu/}).} 
}

\pdfoutput=1  

\usepackage[usenames,dvipsnames]{color}
\usepackage[colorlinks,bookmarksopen,bookmarksnumbered,citecolor=BrickRed,urlcolor=RoyalPurple,linkcolor=BrickRed]{hyperref}

\newlength{\noteWidth}
\long\def\notes#1{\ifinner
           {\footnotesize #1}
           \else
           \marginpar{\parbox[t]{\noteWidth}{\raggedright\footnotesize #1}}
       \fi\typeout{#1}}

\def\notes#1{\typeout{read notes: #1}}  

\def\spm#1{\notes{SPM:  #1}}

\def\rd#1{{\color{red}#1}}

\def\urls#1{{\small \url{#1}}}

\def\util{{\mathcal{U}}}

\def\meanutil{\mbox{\scriptsize$\widebar{\cal U}$}}

\def\reward{w}
\def\rewardDer{W}
\def\rewardDerOpt{\check{W}}
\def\rewardMean{\mbox{\small$\widebar{W}$}}
\def\rewardMean{\widebar{w}}

\makeatletter
\newdimen\rh@wd
\newdimen\rh@hta
\newdimen\rh@htb
\newbox\rh@box
\def\rh@measure#1{\setbox\rh@box=\hbox{$#1$}\rh@wd=\wd\rh@box \rh@hta=\ht\rh@box}

\def\widecheck#1{\rh@measure{#1}%
  \setbox\rh@box=\hbox{$\widehat{\vrule height \rh@hta width\z@ \kern\rh@wd}$}%
  \rh@htb=\ht\rh@box \advance\rh@htb\rh@hta \advance\rh@htb\p@
  \ooalign{$\vrule height \ht\rh@box width\z@ #1$\cr
           \raise\rh@htb\hbox{\scalebox{1}[-1]{\box\rh@box}}\cr}}
\makeatother

\def\rewardDerOpt{\check{W}}

\def\xz{x^\circ}

\def\EFn#1{\Lambda_{#1}}

\def\preclH{{\cal H}^\circ}
\def\preH{H^\circ}  

\def\sq{\hbox{\rlap{$\sqcap$}$\sqcup$}}
\def\qed{\ifmmode\sq\else{\unskip\nobreak\hfil
\penalty50\hskip1em\null\nobreak\hfil\sq
\parfillskip=0pt\finalhyphendemerits=0\endgraf}\fi\medskip}

\long\def\defbox#1{\framebox[.9\hsize][c]{\parbox{.85\hsize}{%
\parindent=0pt
\baselineskip=12pt plus .1pt      
\parskip=6pt plus 1.5pt minus 1pt 
 #1}}}

\long\def\beginbox#1\endbox{\subsection*{}%
\hbox{\hspace{.05\hsize}\defbox{\medskip#1\bigskip}}%
\subsection*{}}

\def\endbox{}

\def\transpose{{\hbox{\it\tiny T}}}

\newsavebox{\junk}
\savebox{\junk}[1.6mm]{\hbox{$|\!|\!|$}}

\def\liminf{\mathop{\rm lim\ inf}}
\def\argmin{\mathop{\rm arg\, min}}

\def\argmax{\mathop{\rm arg\, max}}

\def\A{{\sf A}}

\def\S{{\sf S}}
\def\U{{\sf U}}

\def\state{{\sf X}}
\def\stateu{{\sf X}_{\sf u}}
\def\staten{{\sf X}_{\sf n}}
\def\bx{\clB(\state)}

\newcommand{\field}[1]{\mathbb{#1}}

\def\Re{\field{R}}

\def\One{\mbox{\rm{\bf\large{1}}}} 
\def\One{\field{I}}

\def\ind{\field{I}}

\def\cpi{\check{\pi}} 
\def\cP{{\check{P}}}

\def\cR{{\check{R}}}

\def\bfmath#1{{\mathchoice{\mbox{\boldmath$#1$}}%
{\mbox{\boldmath$#1$}}%
{\mbox{\boldmath$\scriptstyle#1$}}%
{\mbox{\boldmath$\scriptscriptstyle#1$}}}}

\def\bfmN{\bfmath{N}}

\def\bfmU{\bfmath{U}}

\def\bfmX{\bfmath{X}}

\def\bfmY{\bfmath{Y}}

\def\bfmhhaY{\bfmath{\hhaY}} 
\def\bfmhhaY{\hbox to 0pt{$\widehat{\bfmY}$\hss}\widehat{\phantom{\raise 1.25pt\hbox{$\bfmY$}}}}

\def\bfPhi{\bfmath{\Phi}}

\def\til={{\widetilde =}}

\def\clB{{\cal B}}

\def\clF{{\cal F}}

\def\clK{{\cal K}}

\def\clV{{\cal V}}

 \def\FRAC#1#2#3{\genfrac{}{}{}{#1}{#2}{#3}}

\def\ddtp{{\mathchoice{\FRAC{1}{d^{\hbox to 2pt{\rm\tiny +\hss}}}{dt}}%
{\FRAC{1}{d^{\hbox to 2pt{\rm\tiny +\hss}}}{dt}}%
{\FRAC{3}{d^{\hbox to 2pt{\rm\tiny +\hss}}}{dt}}%
{\FRAC{3}{d^{\hbox to 2pt{\rm\tiny +\hss}}}{dt}}}}

\def\ddzeta{{\mathchoice{\FRAC{1}{d}{d\zeta}}%
{\FRAC{1}{d}{d\zeta}}%
{\FRAC{3}{d}{d\zeta}}%
{\FRAC{3}{d}{d\zeta}}}}

\def\half{{\mathchoice{\FRAC{1}{1}{2}}%
{\FRAC{1}{1}{2}}%
{\FRAC{3}{1}{2}}%
{\FRAC{3}{1}{2}}}}

\def\eqdef{\mathbin{:=}}

\def\Prob{{\sf P}}

\def\Expect{{\sf E}}

\def\average#1,#2,{{1\over #2} \sum_{#1}^{#2}}

\def\eye(#1){{\bf(#1)}\quad}
\def\epsy{\varepsilon}

\def\varble{\,\cdot\,}

\newtheorem{theorem}{Theorem}[section]

\newtheorem{proposition}[theorem]{Proposition}
\newtheorem{lemma}[theorem]{Lemma}

\def\Lemma#1{Lemma~\ref{#1}}
\def\Prop#1{Prop.~\ref{#1}}
\def\Theorem#1{Theorem~\ref{#1}}

\def\Section#1{Section~\ref{#1}}

\def\barkappa{{\overline{\kappa}}}

\newcounter{rmnumx}
\newenvironment{romannumx}{\begin{list}{{\upshape (\roman{rmnumx})}}{\usecounter{rmnumx}
\setlength{\leftmargin}{2pt}
\setlength{\rightmargin}{4pt}
\setlength{\itemsep}{3pt}
\setlength{\itemindent}{18pt}
}}{\end{list}}

\newcounter{anum}

\def\Ebox#1#2{%
\begin{center}
\includegraphics[width= #1\hsize]{#2} \end{center}}

\graphicspath{{figures/}}

\def\Fig#1{Fig.~\ref{#1}}

\def\ind{\field{I}}

\def\Re{\field{R}}

\begin{document}

\maketitle

\begin{abstract} 
A new approach to computation of optimal policies for  MDP (Markov decision process)
models is introduced.  The main idea is to solve not one, but an entire family of MDPs, parameterized by a scalar $\zeta$ that appears in the one-step reward function.   For an MDP with $d$ states,  the family of value functions $\{ h^*_\zeta : \zeta\in\Re\}$ is the solution to an ODE,
$$
\ddzeta h^*_\zeta  =  {\cal V}(h^*_\zeta)
$$
where the vector field ${\cal V}\colon\Re^d\to\Re^d$ has a simple form,  based on a  matrix inverse.  

This general methodology is applied to a family of average-cost optimal control models in which the one-step reward function is defined by Kullback-Leibler divergence. 
The motivation for this reward function in prior work is computation:  The solution to the MDP can be expressed in terms of the Perron-Frobenius eigenvector for an associated positive matrix.   The drawback with this approach is that no hard constraints on the control are permitted.   
It is shown here that it is possible to extend this framework to model randomness from nature that cannot be modified by the controller.   Perron-Frobenius theory is no longer applicable -- the resulting dynamic programming equations appear as complex as a completely unstructured MDP model.  Despite this apparent complexity,  it is shown that this class of MDPs admits a solution via this ODE technique. 
This approach is new and practical even for the simpler problem in which randomness from nature is absent.

\smallbreak



\end{abstract}


\section{Introduction} 
\label{s:intro} 

This paper concerns average-cost optimal control for Markovian models.  It is assumed that there is a one-step reward   $\reward$ that is a function of state-input pairs.  For a given policy that defines the input as a function of present and past state values,  the resulting average reward is the limit infimum,
\begin{equation}
\eta = \liminf_{T\to\infty} \frac{1}{T} \sum_{t=1}^T \reward(X(t), U(t))
\label{e:eta}
\end{equation}
where $\bfmX=\{X(t): t\ge 0\}$, $\bfmU=\{U(t): t\ge 0\}$ are the state and input sequences.   Under general conditions, the maximum over all policies is deterministic and independent of the initial condition, and the optimal policy is state-feedback --- obtained as the solution to the average-reward optimality equations  (AROE)  \cite{bershr96a,put14}.

\subsection{Background}

In this paper the state space $\state$ on which $\bfmX$ evolves is taken to be finite, but possibly large.  It is well known that computation of a solution to the AROE may be difficult in such cases.  This is one motivation for the introduction of approximation techniques such as reinforcement learning  \cite{bertsi96a,sutbar98}.
\spm{q-learning isn't really approximate, so I removed}

An interesting alternative is to change the problem so that it is easily solved.  In all 
of the prior work surveyed here, the setting is average-cost optimal control,  so that the reward function is replaced by a cost function.    

Brockett in \cite{bro08} introduces a class of controlled Markov models in continuous time.  The model and cost function are formulated so that the optimal control problem is  easily solved numerically.  The ODE technique is applied to this special class of MDPs  in \Section{s:Brockett}.

\spm{I just learned of this guy,
Krishnamurthy Dvijotham
\\
Postdoctoral Fellow, Center for Mathematics of Information, California Institute of Technology
\\
optimal control, convex optimization, power systems, smart grids
\\
I must have met him before.  He will be at IMA in May!  }

The theory developed in \Section{s:design} was inspired by the work of  Todorov~\cite{tod07},  the similar earlier work of K\'arn\'y~\cite{kar96},  and the more recent work \cite{guaragwil14,meybarbusyueehr15}.   The state space $\state$ is finite,  the action space 
$\U$ consists of probability mass functions on $\state$, and the controlled transition matrix is entirely determined by the input as follows:
\begin{equation}
\Prob\{X(t+1)=x'\mid X(t) =x, U(t) = \mu\} = \mu(x')\,,\qquad x,x\in\state,\ \mu\in\U\, .
\label{e:TodMDPu}
\end{equation}
The MDP has a simple solution only under special conditions on the cost function.  It is assumed
in \cite{tod07,guaragwil14,meybarbusyueehr15} that it is the sum two terms:  The first is a cost function on $\state$, which is completely unstructured.   The second term is a ``control cost'',  defined using   Kullback--Leibler  (K-L) divergence (also known as \textit{relative entropy}).   


The control cost is based on deviation from control-free behavior
(modeled by a nominal transition matrix $P_0$).  In most applications, $P_0$ captures randomness from nature.  For example, in a queueing model  there is uncertainty in inter-arrival times or service times.  An optimal solution in this framework would allow modification of arrival statistics and service statistics,  which may be entirely infeasible.  In this paper the K-L cost framework is broadened to include constraints on the pmf $\mu$ 
appearing in \eqref{e:TodMDPu}.

 \notes{
\rd{It might be fun to reference recent work on randomized control architectures:
\\
 Wells, D.K., Kath, W.L., and Motter, A.E. (2015). Control of stochastic and induced switching in biophysical complex networks. Phys. Rev. X, 5, 031036.    Basis of this survey:
 \\
\urls{https://sinews.siam.org/DetailsPage/tabid/607/ArticleID/713/Leveraging-Noise-to-Control-Complex-Networks.aspx}
 \\
Leveraging Noise to Control Complex Networks, 
\textit{SIAM News}.  Daniel K. Wells, William L. Kath, Adilson E. Motter.  January 19, 2016.  
\\
postscript, March 2016:  the paper looks like fluff to me now:(
 }
 }
 
\subsection{Contributions}

The new approach to computation proposed in this paper is based on the solution of an entire family of MDP problems.     \Section{s:ode} begins with a general MDP model in which the one-step reward function is a smooth function of a real parameter $\zeta\in\Re$.   

For each $\zeta$, the solution to the average-reward optimization problem is based on a relative value function $h^*_\zeta\colon\state\to\Re$.  Under general conditions it is shown that these functions are obtained as the solution to an ordinary differential equation
\[
\ddzeta h^*_\zeta  =  \clV(h^*_\zeta)
\]
Consequently,   \textit{the solution to an entire family of MDPs can be obtained through the solution of a single ordinary differential equation {\rm (ODE)}.}   

Following the presentation of these general results, the paper focuses on the  class of MDPs with transition dynamics given in \eqref{e:TodMDPu}:   the input space is a subset of the simplex in $\Re^d$, and the cost function $c$ is defined with respect to K-L divergence (see \eqref{e:DVrate} and surrounding discussion).   The optimal control formulation is far more general than in the aforementioned work \cite{tod07,guaragwil14,meybarbusyueehr15}, as it allows for inclusion of exogenous randomness in the MDP model.


The dynamic programming equations become significantly  more complex in this generality, so that in particular, the Perron-Frobenious computational approach used in  prior work is no longer applicable.   
Nevertheless, the ODE approach can be applied to solve the family of MDP optimal control problems.
The vector field $\clV\colon\Re^d\to\Re^d$ has special structure that further simplifies computation of the relative value functions.


 Simultaneous computation of the optimal policies is essential in applications to ``demand dispatch'' for providing virtual energy storage from a large collection of flexible loads \cite{barbusmey14,meybarbusyueehr15,busmey14}.   In these papers, randomized policies are designed for each of many thousands of electric loads in a distributed control architecture.  In this context it is necessary to compute the optimal transition matrix $\cP_\zeta$ for each $\zeta$.  
Prior to the present work it was not possible to include any  exogenous uncertainty in the load model.   

In the companion paper  \cite{busmey16v},   the results of the present paper are applied to distributed control of flexible loads, including thermostatically controlled devices such as refrigerators and heating systems.  This paper also contains extensions of the ODE method to generate transition matrices with desirable properties, without consideration of optimality.   

\notes{Let's leave this out of the initial submission -- might piss off a mathematician! This research is a foundation for the recent patent application \cite{busmeyPatent15}.}

 \medskip
 
The remainder of the paper is organized as follows.   \Section{s:ode}
sets the notation for the MDP models in a general setting, and presents
an ODE approach to solving the AROE under minimal conditions on the model.
\Section{s:design} describes the Kullback--Leibler cost criterion.  Special structure of optimal policies obtained in \Theorem{t:IPD} leads to a simple representation of the ODE in
\Theorem{t:IDPODE}.  Conclusions and topics for future research are contained in \Section{s:conc}.

\section{ODE for MDPs} 
\label{s:ode}

\subsection{MDP model}

Consider an MDP with finite state space $\state=\{x^1,\dots, x^d\}$;  the action space $\U$  is an open subset of   $\Re^m$.  The state process is denoted $\bfmX = (X(0),X(1), \dots )$,
 and the input process $\bfmU = (U(0), U(1),\dots)$.  The dynamics of the model are defined by a 
 \textit{controlled transition matrix}:    for  $x,x'\in\state$, and $u\in\U$, this is defined by
\[
 P_u(x,x') = \Prob\{ X(t+1)  =x' \mid X(t)=x,\ U(t)=u \} 
\] 
where the right hand side is assumed independent of $t=0,1,2,\dots$.

The one-step reward function is parameterized by a scalar $\zeta\in\Re$.  It is assumed to be continuously differentiable in this parameter, with derivative denoted
\begin{equation}
\rewardDer_\zeta(x,u) = \ddzeta \reward_\zeta(x, u)\, .
\label{e:WelfareDerivative}
\end{equation} 
Unless there is risk of confusion,  dependency on $\zeta$ will be suppressed;  in particular, we write $\reward$ rather than $\reward_\zeta$.

There may be hard constraints:  For each $x\in\state$, there is an open set $\U(x)\subset\U$ consisting of feasible inputs $U(t)$ when $X(t)=x$.

The optimal reward $\eta^*$ is defined to be the maximum of $\eta$ in \eqref{e:eta} over all policies. 
Under general conditions on the model,  $\eta^*$ is deterministic, and is independent of $x$. Under further conditions, this value and the optimal policy are characterized by
the AROE:
\begin{equation}
\max_{u\in\U(x)}\Bigl\{ \reward(x, u) +\sum_{x'} P_u(x,x') h^*(x') \Bigr\} = h^*(x) + \eta^*
\label{e:AROEgen}
\end{equation}   
in which the function $h^*\colon\state\to\Re$ is called the \textit{relative value function}. 
The stationary policy $\phi^*$ is obtained from the AROE:
  $\phi^*(x)\in\U$ is a maximizing value of $u$ in \eqref{e:AROEgen} for each $x$ \cite{bershr96a,put14}.

Structure for the optimal average reward is obtained under minimal assumptions:
\begin{proposition}
\label{t:etaConvex}
Suppose that the following hold:
\begin{romannumx}
\item 
The welfare function is affine in its parameter:  $\reward_\zeta(x,u) = \reward_0(x,u)  + \zeta \rewardDer(x,u)$ for some function $\rewardDer$ and  all $x,u$.  

\item
For each $\zeta$, the optimal reward $\eta^*_\zeta$ exists, is deterministic, and is  independent of the initial condition. 

\item
For each $\zeta$, the optimal reward  $\eta^*_\zeta$  is achieved with a stationary policy $\phi_\zeta^*$,  and  under this policy, the following ergodic limits exist for each initial condition:
\[
\eta^*_\zeta =   \lim_{T\to\infty} \frac{1}{T} \sum_{t=1}^T \reward_\zeta(X(t), U(t))\,  ,
\qquad
\rewardMean_\zeta =   \lim_{T\to\infty} \frac{1}{T} \sum_{t=1}^T \rewardDer(X(t), U(t))
\]
\end{romannumx}
Then, $\eta^*_\zeta$ is convex as a function of $\zeta$,  with sub-derivative $\rewardMean_\zeta$:  
\[
\eta^*_\zeta \ge \eta^*_{\zeta_0}  + (\zeta-\zeta_0)   \rewardMean_{\zeta_0},\qquad  
\text{for all} \ \zeta,\zeta_0\in\Re\, .
\]
\end{proposition}

\begin{proof}
Convexity of $\eta^*_\zeta$  will follow from the lower bound.  Alternatively, convexity is  implied by the linear programming representation of average-cost optimal control, where $\eta^*_\zeta$ is
 defined as the maximum of linear functions of $\zeta$ \cite{man60a,bor02a}.
 
  To obtain the lower bound, choose any $\zeta,\zeta_0\in\Re$,  and consider the average reward based on $\reward_\zeta$, obtained using $U(t) = \phi^*_{\zeta_0}(X(t))$ for all $t\ge 0$.  We then have,
\[
\begin{aligned}
\eta^*_\zeta &\ge  \liminf_{T\to\infty} \frac{1}{T} \sum_{t=1}^T \reward_\zeta(X(t), U(t))
\\
&=\lim_{T\to\infty} \frac{1}{T} \sum_{t=1}^T \reward_{\zeta_0}(X(t), U(t))
 + \lim_{T\to\infty} \frac{1}{T} \sum_{t=1}^T\bigl(  \reward_\zeta(X(t), U(t))- \reward_{\zeta_0}(X(t), U(t))\bigr)
\end{aligned}
\]
The first summation on the right hand side is equal to $\eta^*_{\zeta_0}  $.  The second reduces to 
$ (\zeta-\zeta_0)   \rewardMean_{\zeta_0}$ on  substituting
$ \reward_\zeta - \reward_{\zeta_0} =(\zeta-\zeta_0) \rewardDer$.
\end{proof}

We next introduce an ODE that solves the AROE for each $\zeta$.

\subsection{ODE solution}

To construct an ordinary differential equation for $h^*_\zeta$ requires several assumptions.   The first is a normalization:  
The relative value function is  not unique, since we can add a constant to obtain a new solution. 
We resolve this by fixing a state $\xz\in\state$, and assume that $h^*_\zeta(\xz)=0$ for each $\zeta$.

For any  function $h\colon\state\to \Re$, we define a new function on $\state\times\U$ via,
\[
P_u h\, (x) = \sum_{x'} P_u(x,x') h(x') 
\]
Similar notation is used for an uncontrolled transition matrix.

\smallbreak

\noindent
\textbf{Assumptions}
\nobreak
\begin{romannumx}

\item For each $\zeta$,  a solution to the AROE   $(h_\zeta^*,\eta_\zeta^*)$ exists, with $h_\zeta(\xz)=0$, and the pair is continuously differentiable in $\zeta$.   Moreover, the function of $(x,u,\zeta)$ defined by,
\[
q_\zeta^* (x,u) = \reward_\zeta(x,u) +   P_u h^*_\zeta\, (x)
\]
is jointly continuously differentiable in $(\zeta,u)$, with the representation
\begin{equation}
\ddzeta q_\zeta^*(x,u) = \rewardDer_\zeta(x,u) + P_u H^*_\zeta\, (x)
\quad
\text{\it in which }
\quad
H^*_\zeta(x) =\ddzeta h^*_\zeta\, (x)\,.
\label{e:qDer}
\end{equation}

\item  The  stationary policy exists as the minimum
\[
\phi^*_\zeta(x)=\argmin_{u\in\U(x)} 
q_\zeta^* (x,u) \,,\qquad x\in\state,
\]
  and is continuously differentiable in $\zeta$ for each $x$.

\item The optimal transition matrix $\cP_\zeta$ is irreducible, with unique invariant pmf denoted $\pi_\zeta$, 
where  
\[
\cP_\zeta(x,x') = P_{u^*}(x,x'),\qquad u^* = \phi^*_\zeta(x),\ x,x'\in\state \, .
\]
\end{romannumx}
All of these assumptions hold for the class of MDP models considered in \Section{s:design}. 

These assumptions imply that for each $\zeta$ there is a solution $H_\zeta$ to \textit{Poisson's equation}, 
\begin{equation}
   \rewardDerOpt_\zeta+ \cP_\zeta H_\zeta  = H_\zeta +  \rewardMean_\zeta  
\label{e:PoissonGen}
\end{equation}
in which  $\rewardDerOpt_\zeta(x)=\rewardDer_\zeta(x, \phi^*_\zeta(x))$, and
 $\rewardMean_\zeta = \sum_x \pi_\zeta(x) \rewardDerOpt_\zeta(x)$.   It is assumed throughout that the solution is normalized, with  $H_\zeta(\xz)=0$;   there is a    unique solution to  \eqref{e:PoissonGen} with this normalization \cite[Thm.~17.7.2]{MT}.

The function $q_\zeta^*$ is the ``Q-function'' that appears in Q-learning
\cite{bertsi96a}.
Under (i) and (ii), it follows from the AROE that  for each $x$ and $\zeta$,
\begin{equation}
q_\zeta^* (x,\phi^*_\zeta(x))
=
\min_{u\in\U(x)} \bigl\{  \reward_\zeta(x, u) +   P_u h^*_\zeta\, (x) \bigr\}
  =
 h^*_\zeta(x) + \eta^*_\zeta
\label{e:AROE-Poisson}
\end{equation}
 
 \subsection*{MDP vector field} 
In \Theorem{t:ODEforMDP} it is shown that the family of relative value functions solves an ODE.  A function $h\colon\state\to\Re$ is regarded as a vector in $\Re^d$.  The vector field $\clV$ is not homogeneous, so it is regarded as a mapping  $\clV\colon\Re^{d+1}\to\Re^d$.  
For  a given   function $h\colon\state\to\Re$ and $\zeta\in\Re$,  the function $\clV(h,\zeta)$ is defined through the following steps:  
\begin{romannumx}
\item[1.] Obtain a policy:
$
\phi(x) = \argmax_u\{ \reward_\zeta(x,u) +P_uh\, (x) \}
$. 
\item[2.]  Obtain a transition matrix $\cP(x,x') = P_{\phi(x)} (x,x')$,  \ \ $x,x'\in\state$.
\item[3.]  Obtain the solution to Poisson's equation,  $\rewardDerOpt + \cP H = H +\rewardMean$, in which $\rewardDerOpt(x) = \rewardDer_\zeta(x,\phi(x))$, $x\in\state$,  and  $\rewardMean$ is the steady-state mean of $\rewardDerOpt$ under this policy.  The solution is normalized so that $H(\xz)=0$.
\item[4.]  Set $\clV(h,\zeta)=H$.  
\end{romannumx}

\begin{theorem}
\label{t:ODEforMDP}
Under the assumptions of this section:
\begin{romannumx}
\item
The family of relative value functions $\{h^*_\zeta \}$ is a solution to the ordinary differential equation,
\begin{equation}
\ddzeta h^*_\zeta = \clV(h^*_\zeta,\zeta)
\label{e:ODEmain}
\end{equation}

\item  $\ddzeta  \eta^*_\zeta  =\rewardMean_\zeta$,
with 
\[
\rewardMean_\zeta = \sum_x
\pi_\zeta(x) \rewardDer_\zeta(x,\phi_\zeta^*(x))\, ,
\qquad
\text{
where $\pi_\zeta$ is the invariant pmf for $\cP_\zeta$.  
}
\]

\item
If the derivative $\rewardDer_\zeta(x,u)$ is independent of $\zeta$ and $u$ for each $x$, 
 then the ODE is homogeneous.  That is, for each $h$, the function $\clV(h,\zeta)$ does not depend on $\zeta $. 
\end{romannumx}
 \end{theorem}

\begin{proof}
The domain of $\clV$ is defined to be any $h$ for which the solution to (1)--(3) is possible.  The domain may not include all functions $h$,  but it is defined for any of the relative value functions $\{h^*_\zeta\}$;  this is true by the assumptions imposed in the theorem.

If $\rewardDer_\zeta$ is independent of $\zeta$ and $u$,
 then  $\reward_\zeta(x,u) = \reward_0(x,u) +\zeta \rewardDer(x)$ for each $x,u,\zeta$.   It follows that $\phi$ is  independent of $\zeta$ in step~1,  and $\rewardDerOpt$ is independent of $\zeta$
in step~3.  Hence the vector field is independent of $\zeta$.

To complete the proof it remains to establish (i),  which will lead to the representation for  $\ddzeta  \eta^*_\zeta $ in part~(ii) of the theorem.  

The assumption that $\U$ is open and that $q_\zeta^*$ is continuously differentiable   is used to apply the first-order condition for optimality of $\phi_\zeta^*(x)$:   
\[
0 = 
\frac{\partial}{\partial u}
q_\zeta^* (x,u) \Big|_{u = \phi^*_\zeta(x)}
\]
On differentiating each side of the AROE in the form \eqref{e:AROE-Poisson},
we obtain from the chain-rule
\[
\begin{aligned}
\ddzeta \Bigl\{  h^*_\zeta(x) + \eta^*_\zeta\Bigr\} 
	& =\ddzeta  \Bigl\{ q_\zeta^* (x,  \phi^*_\zeta(x)) \Bigr\}
\\
	&=
 \frac{\partial}{\partial\zeta} q_\zeta^* \, (x, u)
  \Big|_{u = \phi^*_\zeta(x)}
			+ 
  \Bigl(\frac{\partial}{\partial u}
q_\zeta^* (x,u) \Big|_{u = \phi^*_\zeta(x)} \Bigr) \frac{\partial}{\partial\zeta}  \phi^*_\zeta(x)
\\
	&=
 \frac{\partial}{\partial\zeta} q_\zeta^* \, (x, u)
  \Big|_{u = \phi^*_\zeta(x)} 
  \\
	&=
   \rewardDerOpt_\zeta(x) + P_u H^*_\zeta\, (x)    \Big|_{u = \phi^*_\zeta(x)} 
\end{aligned}
\] 
where in the last equation we have applied \eqref{e:qDer}.  Rearranging terms leads to the fixed point equation
\[
   \rewardDerOpt_\zeta + \cP_\zeta H_\zeta^*  = H_\zeta^* +  \ddzeta  \eta^*_\zeta
\]
Taking the mean of each side with respect to $\pi_\zeta$   implies that $\ddzeta  \eta^*_\zeta  =\rewardMean_\zeta$.  This establishes (ii), and completes the proof that  \eqref{e:ODEmain} holds.
\end{proof}

\subsection*{Extensions}

An ODE can be constructed for the discounted cost problem with discounted factor $\beta\in (0,1)$.  The DROE (discounted-reward optimality equation) is another fixed point equation, similar to \eqref{e:AROEgen}:
\[
\max_{u\in\U(x)}\Bigl\{ \reward(x, u) + \beta \sum_{x'} P_u(x,x') h^*(x') \Bigr\} = h^*(x) 
\]
Step~3 in the construction of the vector field for the ODE is modified as follows: 
\textit{Obtain the solution to  $\rewardDerOpt + \beta\cP H = H $.}  The solution is unique, so no normalization is possible (or needed).
 
For both average- and discounted-reward settings,  an ODE can be constructed when $\U$ is a \textit{finite set} rather than an open subset of $\Re^m$.  In this case, under general conditions, the vector field $\clV$ is continuous and piecewise smooth, and
the optimal policy is piecewise constant as a function of $\zeta$.

\medskip

We next consider two simple examples to illustrate the conclusions in the average cost setting.  The ODE is homogeneous in all of the examples that follow.

\subsection{Example 1:  Linear-quadratic model}
 
Consider first the simple scalar linear system,
\[
X(t+1) = \alpha X(t) + U(t) + N(t+1)
\]
in which $0<\alpha<1$.   The disturbance $\bfmN$ is i.i.d.\ with zero mean, and finite variance $\sigma^2_N$.  
The state space and action space are the real line, $\state=\U=\Re$.  The reward function is taken to be quadratic in the state variable,  $\reward(x,u) = -\zeta x^2 - c(u)$,  so that $\rewardDer(x,u) = -x^2$ is independent of both $\zeta$ and $u$.
The cost function  $c\colon\Re\to\Re$ is continuously differentiable and convex, its derivative $  c'$ is globally Lipschitz continuous, and $c(0) = c'(0)=0$.   

It is assumed in this example that $\zeta\ge 0$.  It can be shown that the relative value function $h^*_\zeta$ is a concave function of $x$ under these assumptions;  it is normalized so that $h^*_\zeta(0)=0$ for each $\zeta\ge 0$  (that is,   $\xz=0$).
\spm{I can prove concavity!   The proof isn't difficult, but a distraction.  Should we remove the claim if I can't find a reference?  Or I could say the proof is similar to \cite{huachemehmeysur11}}

The ODE can be developed even in this infinite state-space setting.

Notation is simplified if this is converted to an average-cost optimization problem, with one-step cost  function $c_\zeta(x,u) = \zeta x^2 + c(u)$.   We let $g_\zeta^*=-h_\zeta^*$, which is a convex  function on $\Re$.  The AROE becomes the average-cost optimality equation,
 \begin{equation}
\min_u\bigl\{ c_\zeta(x, u) +  P_ug^*_\zeta(x) \bigr\} = g^*_\zeta(x) + \gamma^*_\zeta
\label{e:ACOEgen}
\end{equation}   
with $\gamma^*_\zeta=-\eta^*_\zeta$.  The optimal policy is the minimizer,
\[
\begin{aligned}
\phi^*_\zeta(x) & =
\argmin_u\bigl\{ c(u) + P_ug^*_\zeta(x) \bigr\}  
\\
&=
\argmin_u\bigl\{ c(u) + \Expect[g^*_\zeta(\alpha x +u+N_1)] \bigr\}  
\end{aligned}
\]

The ODE is modified as follows.  Let $\clK$ denote the set of non-negative convex functions $g\colon\Re\to\Re$, and construct the vector field so that  $\clV\colon\clK\to\clK$.   For given $g\in\clK$,  we must define $\clV(g)$.  Since we are minimizing cost,  step~1 in the construction of the vector field becomes,

Obtain a policy:
$\displaystyle
\phi(x) = \argmin_u\bigl\{ c(u) + \Expect[g(\alpha x +u+N_1)] \bigr\}
$.
\\
This is a convex optimization problem whose solution can be obtained numerically. 

Step~2  is obtained as follows:
\[
\cP(x,A) = \Prob\{\alpha x +\phi(x)+N_1\in A\} ,\qquad x\in\state,\ A\in\bx.
\]
 The solution to Poisson's equation in step~3 of the ODE construction is more challenging.  This might be approximated using a basis, as in TD-learning \cite{bertsi96a,sutbar98}.

\medskip

 \begin{figure}[h]
\Ebox{.75}{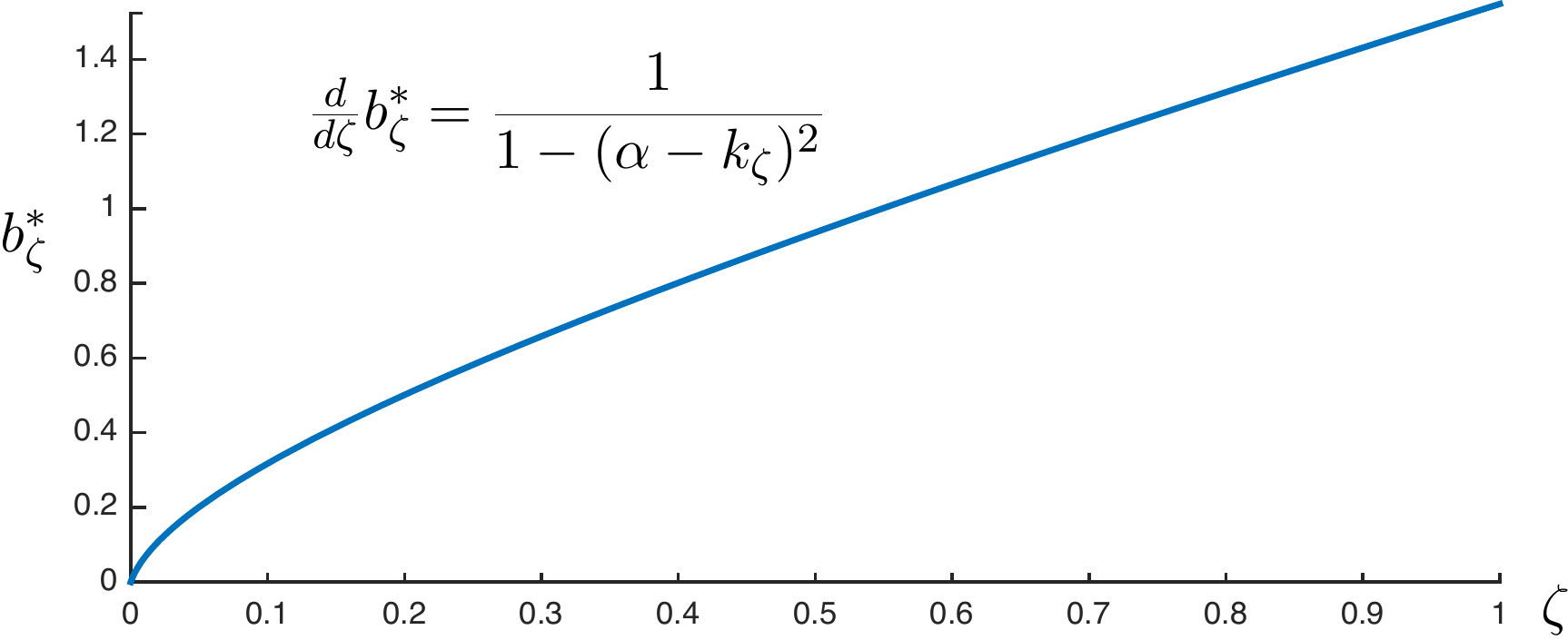} 
\vspace{-2.5ex}
\caption{Solution to the ODE \eqref{e:ODElinear}
 for the linear-quadratic example.}
\label{f:LQRexample}
\vspace{-1.25ex}
\end{figure} 

If the control cost is quadratic, $c(u)=u^2$, then the relative value function is also quadratic,
so that $g^*_\zeta(x) =-h^*_\zeta(x) = bx^2$, with $b\ge 0$,  and $b=0$ only if $\zeta=0$.   The optimal policy is linear, $u=-kx$ for some gain $k$.   
The vector field in the ODE can be restricted to functions of this form:  For any $b\ge 0$,
\begin{romannumx}
\item Obtain a policy 
\[
\begin{aligned}
\phi(x) 
	= \argmin_u\bigl\{ u^2 + b\Expect[(\alpha x +u+N_1)^2] \bigr\}
=\argmin_u\bigl\{ u^2 + b [(\alpha x +u)^2 + \sigma_N^2] \bigr\}
\end{aligned}
\]
This gives $u=-kx$,  with 
$k=b\alpha/(1+b)$.

\item  With a linear policy we obtain,
$\cP(x,A) =\Prob\{ (\alpha - k) x+ N_1\in A\}$.
\item  Obtain the solution to Poisson's equation,  $\rewardDer + \cP H = H +\rewardMean$, in which $\rewardMean$ is the steady-state mean of $
\rewardDer$ under this policy.   Since $\rewardDer(x,u) =-x^2$ is quadratic, it follows that $H$
is also quadratic,  $H(x) = -B x^2$, with
\[
-B=\frac{1}{1-(\alpha-k)^2}\,, \quad \text{and $k$ given in (i)}
\]
\item  Set $\clV(g)=-H$. 
\\ That is,  $\ddzeta g^*_\zeta (x) =  B x^2$, $x\in\Re$, $\zeta\ge 0$.
\end{romannumx}
The ODE reduces to a differential equation for the coefficient $b=b^*_\zeta$: 
\[
\ddzeta b^*_\zeta = \frac{1}{1-(\alpha-k_\zeta)^2},\qquad k_\zeta=b^*_\zeta\alpha/(1+b^*_\zeta)
\]
On substitution this simplifies to,
\begin{equation}
\ddzeta b^*_\zeta = \frac{1}{1-(\alpha/(1+b^*_\zeta))^2} 
\label{e:ODElinear}
\end{equation}
with boundary condition ${b^*_\zeta\mid_{\zeta=0}} =0$.
\Fig{f:LQRexample} shows the solution to this ODE for $\zeta\in[0,1]$ with $\alpha=0.95$.

\subsection{Example 2:  Brockett's MDP}
\label{s:Brockett}

\def\generate{{\cal A}}
\def\cgenerate{\check{\cal A}}

\def\Agenerate{{A}}
\def\Acgenerate{\check{A}}

The theory in this paper is restricted to models in discrete time, but the main results carry  over to the continuous time setting.  An illustration is provided using a class of MDP models introduced in  \cite{bro08}.

As assumed in the present paper, the state space $\state=\{x^1,\dots, x^d\}$ is finite, and the input evolves on an open subset of of   $\Re^m$.  The dynamics are defined by a controlled rate matrix (or generator),
\begin{equation}
\Agenerate_u(x,x') = A(x,x') + \sum_{k=1}^m u_k B^k(x,x')
\label{e:broRate}
\end{equation}
where $A$ and $\{B^1,\dots,B^m\}$ are $d\times d$ matrices.     It is assumed that the input is defined by state feedback $U(t) = \phi(X(t))$.  The associated controlled rate matrix
\[
\generate(x,x') = \Agenerate_{\phi(x)}(x,x')\,,\qquad x,x'\in\state,
\]
defines the transition rates for $\bfmX$ under this policy: 
\[
\generate(x,x')  =\lim_{t\downarrow 0} \frac{1}{t} \bigl[ \Prob\{ X(t) = x' \mid X(0) = x\}  -  \ind\{ x=x'\} \bigr]
\]


Adapting the notation to the present paper, the cost function is taken of the form
\[
c_\zeta(x,u) = \zeta \kappa(x) +  \half \| u\|^2. 
\]
in which $\kappa\colon\state\to\Re$.
For the continuous time model, the average-cost optimality equation becomes
 \begin{equation}
\min_u\bigl\{ c_\zeta(x, u) +  \Agenerate_u g^*_\zeta\,(x) \bigr\} =   \gamma^*_\zeta
\label{e:ACOEgenCts}
\end{equation}  
in which $ \gamma^*_\zeta$ is the optimal average cost,  $g^*_\zeta$ is the relative value function,  and
\[
  \Agenerate_u g^*_\zeta\,(x) = \sum_{x'}  \Agenerate_u(x,x')  g^*_\zeta(x'),\qquad x\in\state.
\]
It is assumed that  $g^*_\zeta(\xz)= 0$ for some state $\xz$ and all $\zeta$.

The minimizer in \eqref{e:ACOEgenCts} defines the optimal policy $\phi^*(x)$.  For this model and cost function, the minimizer can be obtained by taking the gradient with respect to $u$ and setting this equal to zero to obtain:
\begin{equation} 
\phi^*_k(x) = -  \sum_{x'} B^k(x,x') g^*_\zeta(x')\,,\qquad x\in\state \, .
\label{e:ACOEgenCtsPolicy}
\end{equation}  

The ODE to solve \eqref{e:ACOEgenCts} takes the following steps.   First, if $\zeta=0$ then obviously $\phi^*\equiv 0$ and  $g^*_0\equiv 0$.   This sets the initial condition for the ODE.    For other $\zeta$ we have as before that $G^*_\zeta = \ddzeta g^*_\zeta$ solves Poisson's equation for the generator $\Acgenerate_\zeta$ obtained with policy $\phi^*_\zeta$:
\begin{equation}
\kappa(x) + \sum_j \Acgenerate_\zeta(x,x')  G^*_\zeta(x') = \barkappa_\zeta ,\qquad x\in\state\,,
\label{e:BrockettFish}
\end{equation}
where $\barkappa_\zeta$ is the steady-state mean of $\kappa$ under this policy, and
\[
\Acgenerate_\zeta(x,x')  = \Agenerate_{\phi^*_\zeta(x)} (x,x')  
\]

\smallbreak

The following numerical example from \cite{bro08} will clarify the construction of the ODE.
In this example $m=1$ so that $u$ is scalar-valued,  and $\state=\{1,2,3\}$.  
Denote $B=B^1$, where in this example
\[
A = 
\begin{bmatrix}
-1 & 1 & 0
\\
1 & -2 & 1
\\
0 & 1 & -1
\end{bmatrix}\, ,
\qquad
B = 
\begin{bmatrix}
-1 & 1 & 0
\\
0 & 0 & 0
\\
0 & 1 & -1
\end{bmatrix}
\]
The input is restricted to $\{u\in\Re : u>-1\}$.
For $\zeta>0$, the cost function is designed to penalize the first and third states:   $c(1)=c(3)=3$,  and $c(x^2)=0$.   In \cite{bro08} the case $\zeta=1/2$ is considered, for which it is shown that $\phi^* (1) =\phi^* (3) =  \sqrt{12} - 3$, and $\phi^* (x^2) =0$. 

Written in vector form,  Poisson's equation \eqref{e:BrockettFish} becomes
\begin{equation}
3 b
+
A_\zeta v_\zeta =
\barkappa_\zeta 
e
\label{e:BrockettFishEx}
\end{equation}
in which  $b^\transpose =(1,0,1)$,  $e^\transpose =(1,1,1)$,
$v_\zeta(i) = G^*_\zeta(x^i)$, and
\[
A_\zeta(i,j) = A(i,j) + \phi_\zeta^*(x^i) B(i,j)\,,\qquad 1\le i,j\le 3.
\]

This example is designed to have   simple structure.
From the form of the optimal policy \eqref{e:ACOEgenCtsPolicy}, it follows that $\phi_\zeta^*(2)=0$ for any $\zeta$.   Moreover,  from symmetry of the model it can be shown that  $\phi_\zeta^*(1) =  \phi_\zeta^*(3)$, and $g^*_\zeta(1) = g^*_\zeta(3)$.  We take $\xz=2$ so that $g^*_\zeta(2) = 0$.  Consequently, the  3-dimensional ODE for $g^*_\zeta$ will reduce to a one-dimensional ODE for 
$\xi_\zeta\eqdef g^*_\zeta(1)$.

The expression for the optimal policy \eqref{e:ACOEgenCtsPolicy} also gives
\[
\phi^*_\zeta(1) = -\xi_\zeta  \sum_j B(1,j) b(j)= \xi_\zeta
\]
And, since the second row of $B$ is zero, it follows that  $A_\zeta = A + \xi_\zeta B$.

We have $\ddzeta \xi_\zeta  = G^*_\zeta(1)= v_\zeta(1)$, and Poisson's equation \eqref{e:BrockettFishEx} becomes
\[
3 b
+
v_\zeta(1)
A_\zeta b =
\barkappa_\zeta  e
\]
The first two rows of this vector equation give
\[
\begin{aligned}
 3  + [-1 +\xi_\zeta (-1)  ] v_\zeta(1) &=\barkappa_\zeta
 \\
 0+ 2 v_\zeta(1)   &=\barkappa_\zeta
\end{aligned}
\]
Substituting the second equation into the first gives
\[
\ddzeta \xi_\zeta = v_\zeta(1) = \frac{3}{3+ \xi_\zeta}
\]
On making the change of variables $f_\zeta=3+ \xi_\zeta$ we obtain
\[
\half \ddzeta f_\zeta^2 =
f_\zeta \ddzeta f_\zeta  = 3    ,
\]
whose solution is given by $   f_\zeta^2  =   f_0^2 + 6\zeta $, with $   f_0^2 = 9$.  

In summary,  $\phi^*_\zeta(1) = \xi_\zeta=-3+ f_\zeta$, giving
\[
\phi^*_\zeta(1) =\phi^*_\zeta(3) = -3 + \sqrt{9 + 6\zeta}\,, \quad \phi^*_\zeta(2) =0.
\]
It is necessary to restrict $\zeta$ to  the interval   $(-5/6,\infty)$
to respect the constraint that $\phi^*_\zeta(x)>-1$ for all $x$.
 
Based on the formula $\ddzeta  \gamma^*_\zeta= \barkappa_\zeta$  and the preceding formula $ \barkappa_\zeta =2v_\zeta(1) = 2 \ddzeta \xi_\zeta$, it follows that
\[
 \gamma^*_\zeta = 2\xi_\zeta =  -6 + 2\sqrt{9 + 6\zeta}\,, \qquad \zeta>-5/6\, ;
\]
a concave function of $\zeta$, as predicted by \Prop{t:etaConvex}.

 \smallbreak

%
%

%
%
%
%

\section{MDPs with Kullback--Leibler Cost} 
\label{s:design}

The general results of the previous section are now applied to a particular class of MDP models. 
 
\subsection{Assumptions and Notation}   

The dynamics of the MDP are assumed of the form \eqref{e:TodMDPu},  where the action space consists of a convex subset of probability mass functions on $\state$.   
The welfare function is assumed to be affine in $\zeta$, as assumed  in \Prop{t:etaConvex}.  To maintain notational consistency with prior work \cite{busmey14,meybarbusyueehr15,busmey16v} we denote 
\begin{equation}
\reward_\zeta = \reward_0 + \zeta \util\,, \qquad \zeta\in\Re\,,
\label{e:KLW}
\end{equation}
and assume that $\util\colon\state\to\Re$ is a function of the state only.  
In the notation of \Prop{t:etaConvex}, we have $\rewardDer(x,u)=\util(x)$ for all $x,u$.
Under these conditions, it was shown in
\Theorem{t:ODEforMDP} that the ODE \eqref{e:ODEmain}
 is homogeneous. 

The first term $ \reward_0$ in \eqref{e:KLW}
 is the negative of a control cost.  Its definition
begins with the specification of a transition matrix $P_0$ that describes  nominal (control-free) behavior.   It is assumed to be \textit{irreducible and aperiodic}.  Equivalently, there is  $n_0\ge 1$ such that for each   $x, x'\in\state$,
\begin{equation}
  P^n_0(x,x') >0,\qquad x\in\state,\ n\ge n_0.
\label{e:unichain}
\end{equation}
It follows that $P_0$ admits a unique invariant pmf, denoted $\pi_0$.   

In the MDP model we deviate from this nominal behavior, but restrict to transition matrices satisfying $P(x,\varble)\prec P_0(x,\varble)$ for each $x$.  In fact,  the optimal solutions will be equivalent:
\begin{equation}
P(x,x')>0\Longleftrightarrow P_0(x,x')>0,\qquad \text{for all } x,x'\in\state
\label{e:PprecP}
\end{equation}
Under this condition it follows that $P$ is also irreducible and aperiodic. 

The following representation will be used in different contexts throughout the paper.   Any function $h\colon\state\times\state\to\Re$ is regarded as an unnormalized log-likelihood ratio:  Denote for $x,x'\in\state$,
\begin{equation}
P_h(x,x') \eqdef P_0(x,x')\exp\bigl(  h( x' \mid x )  -  \EFn{h}(x)    \bigr),  
\label{e:Ph-a}
\end{equation}
in which $h( x' \mid x )$ is the value of $h$ at $(x,x')\in\state\times\state$, and $ \EFn{h}(x)$ is the normalization constant,
\begin{equation}
    \EFn{h}(x)     
\eqdef  \log\Bigl( \sum_{x'} P_0(x,x')\exp\bigl(  h(x'\mid  x)     \bigr) \Bigr)
\label{e:EFn}
\end{equation}

For any transition matrix $P$, an invariant pmf  is interpreted as a row vector, so that invariance can be expressed $\pi P=\pi$.   Any function $f\colon\state\to\Re$ is interpreted as a $d$-dimensional column vector,  and we use the standard notation $Pf\, (x) =   \sum_{x'}P(x,x')f(x')$,  $x\in\state$.

The \textit{fundamental matrix} is   the inverse, 
\begin{equation}
Z = [I - P + 1\otimes \pi]^{-1}
\label{e:fundKernGen}
\end{equation}
where $1\otimes \pi$ is a matrix in which each row is identical, and equal to $ \pi$.  
If $P$ is irreducible and aperiodic, then it can be expressed as a power series: 
\begin{equation}
Z = \sum_{n=0}^\infty  [P - 1\otimes \pi]^n
\label{e:Z}
\end{equation}
with $ [P - 1\otimes \pi]^0 \eqdef I$  (the $d\times d$ identity matrix),
and $ [P - 1\otimes \pi]^n  = P^n - 1\otimes \pi$  for $n\ge 1$.  

The \textit{Donsker-Varadhan rate function} is denoted,
\begin{equation}
K(P\| P_0) = \sum_{x,x'} \pi(x) P(x,x')   \log \Bigl(\frac{P(x,x') }{P_0(x,x')} \Bigr)
\label{e:DVrate}
\end{equation}  
Letting $\Pi(x,x') = \pi(x) P(x,x')  $ and $\Pi_0(x,x') =\pi(x) P_0(x,x')$,   we have
\begin{equation}
K(P\| P_0) = D(\Pi \| \Pi_0)
\label{e:KD}
\end{equation}
where $D$ denotes  K-L divergence.  It is called a ``rate function'' because it defines the relative entropy rate between two stationary Markov chains, and appears in the theory of large deviations for Markov chains \cite{konmey05a}.

For the transition matrix $P_h$ defined in \eqref{e:Ph-a},  the rate function can be expressed in terms of  its  invariant pmf $\pi_h$, the bivariate pmf  $\Pi_h(x,x') = \pi_h(x) P_h(x,x')  $, and the log moment generating function \eqref{e:EFn}:
\begin{equation}
\begin{aligned}
K(P_h\| P_0) &= \sum_{x,x'} \Pi_h(x,x')    \bigl[  h( x' \mid x )  -  \EFn{h}(x) \bigr]
\\
&= \sum_{x,x'}\Pi_h(x,x')      h( x' \mid x )  -     \sum_x \pi_h(x) \EFn{h}(x) 
\end{aligned}
\label{e:DVrate-h}
\end{equation}

As in \cite{tod07,guaragwil14,meybarbusyueehr15}, the rate function is used here 
to model the cost of deviation from the nominal transition matrix $P_0$:  the control objective in this prior work  can be cast as
the solution to the convex optimization problem,
\begin{equation}
\eta^*_\zeta =\max_{\pi,P} \bigl \{\zeta \pi(\util) -  K(P\| P_0)  :  \pi   P  = \pi \bigr\}
\label{e:ARobjectiveReward}
\end{equation}
where $\util\colon\state\to\Re$, and the maximum is over all transition matrices.

\spm{We could have a problem with plagiarism in the following discussion -- it is identical to CDC.   We should shorten, and cite our CDC paper}

\paragraph*{Nature \&\ nurture}   

In many applications it is necessary to include a model of randomness from nature along with the randomness introduced by the local control algorithm (nurture).  This imposes additional constraints in the optimization problem \eqref{e:ARobjectiveReward}.

Consider a Markov model in which the full state space  is the cartesian product of two finite state spaces:  $\state= \stateu\times\staten$,   where  $\stateu$ are components of the state that can be directly manipulated through control.   
The ``nature'' components $\staten$ are not subject to direct control.  For example, these variables may be used to model service and arrival statistics in a queueing model,   or uncertainty in terrain in an application to robotics.

 Elements of   $\state$ are   denoted   $x=(x_u,x_n)$.
Any state  transition matrix under consideration is assumed to have the following conditional-independence structure,
\begin{equation}
P(x,x') = R(x, x_u') Q_0(x,x_n') ,\quad x\in\state, x_u'\in\stateu,\ x_n'\in\staten 
\label{e:PQ0R}
\end{equation}
where
$
\sum_{x_u'} R(x, x_u')=\sum_{x_n'}  Q_0(x,x_n') =1
$
for each $x$.
The matrix $Q_0$ is out of our control -- this models dynamics such as the weather.

To underscore the generality of this model, consider a standard  MDP model with finite  state space $\S$, finite action space $\A$,  and controlled transition law $\varrho$.  Letting $\bfPhi$ denote the state process and $\bfmU$ the input process, we have for any two states $s,s'$, and any action $a$, 
\[
\Prob\{ \Phi(t+1) = s' \mid \Phi(t) = s,\   U(t) = a \}  = \varrho(s' \mid s,a)
\]
A randomized policy is defined by a function  $\phi \colon \A\times\S\to [0,1] $ for which $\phi(\varble \mid s)$ is a probability law on $\A$ for each $s\in\S$. 

\begin{proposition}
\label{t:MDPsubsetNatureNurture}
Consider the MDP model with transition law $\varrho$ and randomized policy $\phi$.   For each $t\ge 0$ denote $X_n(t)=\Phi(t)$ and $X_u(t)=U(t-1)$,  where $X(0) = (U(-1),\Phi(0))$ is the initial condition.   Then $\bfmX=(\bfmX_u,\bfmX_n)$ is a Markov chain on $\state=\A\times\S$,  with transition matrix of the form \eqref{e:PQ0R}, where for $ x,x'\in\state$,
\[
Q_0(x,x_n') = \varrho(x_n' \mid x_n,\ x_u),\quad R(x,x_u') = \phi(x_u' \mid x_n), 
\] 
\end{proposition}

\begin{proof}
From the definitions and Bayes' rule,
\[
\begin{aligned}
 \Prob\{ &X_u    (t+1) = x_u',\  X_n(t+1) = x_n'   \mid X(t) = x\}  
 \\
 & =  \Prob\{ X_n(t+1) = x_n' \mid  X_u(t+1) = x_u' ,\   X(t) = x\}   
\Prob\{   X_u(t+1) = x_u'   \mid X(t) = x\}  
 \\
 & =  \Prob\{ \Phi(t+1) = x_n' \mid  U(t) = x_u' ,\   X(t) = x\} 
   \Prob\{    U(t) = x_u'   \mid X(t) = x\}  
\end{aligned}
\]
Recall that $X(t) = (\Phi(t), U(t-1))$.    The pair $(U(t-1),\Phi(t+1) )$ are conditionally independent given $(\Phi(t),U(t))$,  so that the right hand side becomes,
\[
 \Prob\{ \Phi(t+1) = x_n' \mid  U(t) = x_u' ,\   \Phi(t) = x_n\}   \Prob\{    U(t) = x_u'   \mid \Phi(t) = x_n\}  
\]
This establishes the desired result:
\[
\begin{aligned}
 \Prob\{ X_u(t+1) = x_u',\  & X_n(t+1) = x_n'   \mid X(t) = x\}  
 \\
 & = \varrho(x_n' \mid x_n,\ x_u)  \phi(x_u' \mid x_n)\end{aligned}
\]
\end{proof}

\subsection{Optimal control with Kullback--Leibler cost}
\label{s:AROE}

We consider now the optimization problem \eqref{e:ARobjectiveReward}, subject to the structural constraint \eqref{e:PQ0R},  with $Q_0$ fixed.  The maximizer defines a transition matrix that is denoted,
\begin{equation}
\cP_\zeta =\argmax_P \bigl \{\zeta \pi(\util) -  K(P\| P_0)  : \pi   P  = \pi  \bigr\}
\label{e:ARobjective}
\end{equation}
It is shown in \Prop{t:IDPexists}  that this can be cast as a convex program, even when subject to the structural constraint \eqref{e:PQ0R}. 
The optimization variable in the convex program will be taken to be pmfs $\Pi$ on the product space $\state\times\stateu$. 

Define for each $\pi$ and $R$ the pmf on $\state\times\stateu$,   
\begin{equation}
\Pi_{\pi,R}(x,x_u') = \pi(x) R(x,x_u'),\quad  x\in\state, x_u'\in\stateu\, .
\label{e:PiRdef}
\end{equation}
The pmf $\pi$ can be recovered from $\Pi_{\pi,R}$ via $\pi(x) = \sum_{x_u'}\Pi_{\pi,R}(x,x_u') $,  $x\in\state$,  and the matrix $R$ can also be recovered   via $ R(x,x_u') = \Pi_{\pi,R}(x,x_u')/\pi(x)$, provided $\pi(x)>0$.

The following result shows that we can restrict to $R$ for which $  \Pi_{\pi,R}\prec  \Pi_{\pi_0,R_0}$.
\begin{lemma}
\label{t:bK}
For any transition matrix $P$,
\[
K(P\| P_0)<\infty \ \Longleftrightarrow \ \Pi_{\pi,R}\prec  \Pi_{\pi_0,R_0}
\]
\end{lemma}
\begin{proof} 
If $K(P\| P_0)<\infty$ then $P(x,\varble) \prec   P_0(x,\varble)$.   This implies that  $R(x,\varble) \prec   R_0(x,\varble)$ for each $x\in\state$ satisfying $\pi(x)>0$,
and also that $\pi\prec\pi_0$.    
\spm{I am looking for a short proof of this last statement.  I could iterate,   $P^n(x,\varble) \prec   P^n_0(x,\varble)$, but this is messy. Can we just leave the statement as-is?}

Hence, if $K(P\| P_0)<\infty$,  then for each $x$ and $x_u'$,
\[
\begin{aligned}
\pi_0(x) R_0(x,x_u') = 0 \  & \Rightarrow \ \pi(x) R_0(x,x_u') = 0 \\
		& \Rightarrow \ \pi(x) R(x,x_u') = 0    \,,
\end{aligned}
\]
which establishes one implication:  $\Pi_{\pi,R}\prec  \Pi_{\pi_0,R_0}$ whenever
 $K(P\| P_0)<\infty $.

Conversely, if $\Pi_{\pi,R}\prec  \Pi_{\pi_0,R_0}$ then $K(P\| P_0)<\infty $ by the definition of $K$, and the convention $s\log(s)=0$ when $s=0$.  
\end{proof}

\Lemma{t:bK}
 is one step towards the proof of the following convex program representation of  \eqref{e:ARobjectiveReward}:
 
\begin{proposition}
\label{t:IDPexists}  
The objective function in  \eqref{e:ARobjectiveReward} is concave in the variable $\Pi=\Pi_{\pi,R}$, subject to the convex constraints,
\begin{subequations}
\begin{align} 
\null  & 
   \Pi  \   \text{is a pmf on $\state\times \stateu$}
 \label{e:PiRconstraintsA}
  \\
\null & 
  			 \Pi \prec  \Pi_0,\quad \text{with $\Pi_0(x,x_u') = \pi_0(x) R_0(x,x_u')$}
\label{e:PiRconstraintsB}
  \\
\null &
   \sum_x Q_0(x,x_n')  \Pi (x,x_u')
  = \sum_{x_u}\Pi (x',x_u) \quad  \text{for}\ x'=(x_u',x_n')\in\state 
\label{e:PiRconstraintsC}
\end{align}
\end{subequations}
It admits an optimizer  
$
\Pi ^*(x,x_u') =\cpi_\zeta(x)\cR_\zeta(x,x_u')
$,
in which $\cpi_\zeta(x)>0$ for each $x$.  
Consequently,  there exists an optimizer $\cP_\zeta $  for \eqref{e:ARobjective}, 
with invariant pmf   $\cpi_\zeta$.
\end{proposition}

\begin{proof}
We first consider the constraints: (a) is by definition, and (c) is the invariance constraint for $(\pi,P)$. Constraint (b) is without loss of generality,  given \Lemma{t:bK}.
  
Next we turn to the objective function: 
The function to be maximized in \eqref{e:ARobjectiveReward} can be expressed 
\[
\zeta \pi(\util) -  K(P\| P_0) = \sum_x \pi(x)\reward(x,R) 
\]
in which
\begin{equation}
\begin{aligned}
\reward(x,R)  &= \zeta \util(x) 
		-  \sum_{x'}  P(x,x') \log \Bigl(\frac{P(x,x') }{P_0(x,x')} \Bigr)  
		\\
&= \zeta \util(x) 
		-  \sum_{x'_u}  R(x,x'_u) \log \Bigl(\frac{R(x,x'_u) }{R_0(x,x_u')} \Bigr)  
\end{aligned}
\label{e:AROEcost}
\end{equation}
The second equation follows from the assumption that $P$ depends on $R$ through \eqref{e:PQ0R}.   Multiplying each side by $\pi(x)$ and summing over $x$ we obtain  a representation in terms of the variable $\Pi_{\pi,R}$, with $\zeta \pi(\util) -  K(P\| P_0)  =$
\[ 
 \zeta  \sum_{x,x_u'}  \Pi_{\pi,R}(x,x_u')  \util(x)  -  D(\Pi_{\pi,R}\| \Pi_{\pi,R_0})
\] 
The  K-L divergence $D$ is known to be jointly convex in its two arguments  \cite{demzei98a}.   Since $\Pi_{\pi,R_0}$ is a linear function of $ \Pi_{\pi,R}$, this establishes concavity.

The existence of an optimizer follows from the fact that the function to be optimized is continuous as a function of $\Pi_{\pi,R}$,   and the domain of optimization (\ref{e:PiRconstraintsA}--\ref{e:PiRconstraintsC})
 is compact.  
\end{proof}

It is shown in \Theorem{t:IPD} that the optimal value $\eta^*_\zeta$ together with a \textit{relative value function} $h^*_\zeta$ solve the average reward optimization equation (AROE):
\notes{$w(x,R)$ not  clear.   It really should be $w(x,\mu)$}
\begin{equation}
\max_R\Bigl\{
\reward(x,R)
+\sum_{x'} P(x,x') h^*_\zeta(x') \Bigr\} = h^*_\zeta(x) + \eta^*_\zeta
\label{e:AROE}
\end{equation}
Recall that the relative value function is not unique, since  a new solution is obtained by adding a non-zero constant;  the normalization  $h^*_\zeta(\xz)=0$ is imposed, where $\xz\in\state$ is a fixed state.

   The proof of \Theorem{t:IPD} is given in the Appendix.
 
\begin{theorem}
\label{t:IPD}
There exist optimizers $\{\cpi_\zeta, \cP_\zeta : \zeta\in\Re\}$,  and solutions to
the AROE    $\{h^*_\zeta,\eta^*_\zeta: \zeta\in\Re\}$ with the following properties: 
\begin{romannumx}

\item  
The optimizer $\cP_\zeta$ can be obtained from the relative value function $h^*_\zeta$ as follows:  
\begin{equation}
\cP_\zeta(x,x')
\eqdef P_0(x,x')\exp\bigl(  h_\zeta(x'_u\mid  x)  -  \EFn{h_\zeta}(x)    \bigr)
\label{e:Pzeta-a}
\end{equation}
where  for $ 
x\in\state$, $x_u'\in\stateu$,
\begin{equation}
 h_\zeta(x'_u\mid  x)  =  \sum_{x_n'} Q_0(x,x_n') h^*_\zeta(x_u',x_n'),  
\label{e:hmid}
\end{equation}
and $\EFn{h_\zeta}(x) $ is the normalizing constant \eqref{e:EFn} with   $h=h_\zeta$.

\item
$\{\cpi_\zeta, \cP_\zeta, h^*_\zeta,\eta^*_\zeta: \zeta\in\Re\}$ are continuously differentiable in the parameter $\zeta$.
\end{romannumx}
\qed
\end{theorem}

The fact that the domain of optimization   (\ref{e:PiRconstraintsA}--\ref{e:PiRconstraintsC})
 is compact was helpful in establishing the existence of an optimizer.  However,  the results in \Section{s:ode} require that the action space be an \textit{open} set.  To apply the results of \Section{s:ode} we can apply \Theorem{t:IPD}~(i),  which justifies the restriction \eqref{e:PprecP}.  The restricted action space is an open subset of  $\Re^m$ for some $m<d$. 
\spm{should give formula for the derivative of $\eta^*_\zeta$ here?  We have the sub-derivative result in the previous section, and a derivative result later.}

Representations for the derivatives in \Theorem{t:IPD}~(ii),  in particular the derivative of $\EFn{h_\zeta^*}$ with respect to $\zeta$,  lead to a representation for the ODE   used to compute the optimal  transition matrices $\{\cP_\zeta\}$.

\subsection{ODE Solution}
\label{s:IPD}

It is shown here that the assumptions of \Theorem{t:ODEforMDP} hold, and hence the relative value functions $\{h_\zeta^* : \zeta\in\Re\}$ can be obtained as the solution to an ODE.  

At the start of \Section{s:ode} is is assumed that the action space  is an open subset of Euclidean space,  and this assumption is required in  \Theorem{t:ODEforMDP}. This can be imposed without loss of generality since any optimizer satisfies \eqref{e:PprecP}.

It is convenient to generalize the problem slightly here.   Let $\{h_\zeta^\circ : \zeta\in\Re\}$ denote a family of functions on $\state$, continuously differentiable in the parameter $\zeta$.  They are not necessarily relative value functions, but we maintain the structure established in  \Theorem{t:IPD} for the  family of transition matrices.  Denote,
\begin{equation}
h_\zeta (x'_u\mid  x)   =  \sum_{x_n'} Q_0(x,x_n') h^\circ_\zeta(x_u',x_n'),  
\quad 
x\in\state,\ x_u'\in\stateu
\label{e:hmidcirc}
\end{equation}
and then define as in \eqref{e:Ph-a},
\begin{equation}
 P_\zeta(x,x')
\eqdef P_0(x,x')\exp\bigl(  h_\zeta(x'_u\mid  x )  -  \EFn{h_\zeta}(x)    \bigr)
\label{e:Pzeta}
\end{equation} 
The function   $\EFn{h_\zeta}\colon\state\to\Re$ is a normalizing constant, exactly as in \eqref{e:EFn}:
\[
    \EFn{h_\zeta^\circ}(x)     
\eqdef  \log\Bigl( \sum_{x'} P_0(x,x')\exp\bigl(  h_\zeta(x'_u\mid  x)     \bigr) \Bigr)
\]


We begin with a general method to construct  a family of functions $\{h_\zeta^\circ : \zeta\in\Re\}$ based on an ODE.     Using notation similar to  \Theorem{t:ODEforMDP}, the ODE is expressed,  
\begin{equation}
\ddzeta h_\zeta^\circ  =  \clV(h_\zeta^\circ)\,,\qquad \zeta\in\Re,
\label{e:hODE}
\end{equation}
with boundary condition $h_0^\circ\equiv 0$.  A particular instance of the method will result in $h_\zeta^\circ = h_\zeta^*$ for each $\zeta$.

 
 Assumed given is a mapping $\preclH$ from transition matrices to functions on $\state$.   Following this, the vector field $\clV$ is obtained through the following two steps:  For a function $h\colon\state\to\Re$,  
\begin{romannumx}
\item Define a new transition matrix via \eqref{e:Ph-a},
\begin{equation}
P_h(x,x') \eqdef P_0(x,x')\exp\bigl(  h(x_u' \mid x)  -  \EFn{h}(x)    \bigr),\quad x,x'\in\state,
\label{e:Ph}
\end{equation}
in which $  h(x_u' \mid x) = \sum_{x_n'} Q_0(x,x_n') h(x_u',x_n')$,  and $ \EFn{h}(x)  $ is a normalizing constant.

\item  Compute $\preH = \preclH(P_h)$, and define $\clV(h) = \preH$.
It is assumed   that the functional $\preclH$ is constructed so that  $ \preH(\xz)=0$ for any $P$.

\end{romannumx}


 
 \smallbreak

In \cite{busmey16v} the functional  $\preclH$ is designed to ensure desirable properties in the ``demand dispatch'' application that is the focus of that paper.  It is shown here that a particular choice of the function $\preclH$  will provide the solution to the collection of MDPs \eqref{e:ARobjectiveReward}.  Its domain will include only transition matrices that are irreducible and aperiodic.    For  any transition matrix  $P$ in this domain, the fundamental matrix $Z$ is obtained using \eqref{e:fundKernGen}, and then $\preH=\preclH(P)$ is defined as
\begin{equation}
\preH(x) = \sum_{x'} [ Z(x,x')-Z(\xz,x') ] \util (x'),\qquad x\in\state
\label{e:fishP}
\end{equation}
The function $\preH$   is a solution to Poisson's equation,
\begin{equation}
P \preH =\preH -\util +\meanutil
\label{e:fish}
\end{equation}
where $\meanutil $ (also written $\pi(\util)$) is the steady-state mean of $\util$:
\begin{equation}
 \meanutil  \eqdef \sum_x\pi(x) \util(x) 
\label{e:barutil}
\end{equation}

The proof of \Theorem{t:IDPODE}
  is given in the Appendix. 
\begin{theorem}
\label{t:IDPODE}
Consider the ODE 
\eqref{e:hODE}
with boundary condition $h_0^\circ\equiv 0$,
and with  $\preH=\preclH(P)$ defined using \eqref{e:fishP} for each transition matrix $P$
that is irreducible and aperiodic.

The solution to this ODE exists, and the resulting functions $\{ h^\circ_\zeta : \zeta\in\Re\}$    
coincide with the relative value functions $\{h^*_\zeta: \zeta\in\Re\}$.  Consequently,   $\cP_\zeta = P_{h_\zeta}$ for each $\zeta$.
\qed
\end{theorem}

We sketch here the main ideas of the proof of \Theorem{t:IDPODE}.  

The Implicit Function Theorem is used to establish differentiability of the relative value functions and average reward as a function of $\zeta$.
The ODE representation can then be obtained from
\Theorem{t:ODEforMDP}.

The next step is to establish the particular form for the ODE.
The statement of the theorem is equivalent to the  
 representation $  H^*_\zeta = \preclH(\cP_\zeta)$ for each $\zeta$, where $h^*_\zeta$ is the relative value function, 
$\cP_\zeta$ is defined in \eqref{e:ARobjective}, and
\begin{equation}
  H^*_\zeta =\ddzeta h^*_\zeta 
\label{e:Hrep}
\end{equation}
 The first step in the proof of \eqref{e:Hrep} is a fixed point equation that follows from the AROE.  The following identity is given in \Prop{t:IPDfull}:
\begin{equation}
 \zeta \util   + \EFn{h^*_\zeta} = h^*_\zeta + \eta^*_\zeta \, .
\label{e:NNfixedPt}
\end{equation} 
A representation for the derivative of the log moment generating function is obtained in \Lemma{t:LambdaFish},  
\[
\ddzeta  \EFn{h^*_\zeta}\, (x) = \sum_{x'} \cP_\zeta(x,x') H^*_\zeta (x')\, .
\] 
Differentiating each side of \eqref{e:NNfixedPt} then gives,
\begin{equation}
   \util   +  \cP_\zeta H^*_\zeta= H^*_\zeta + \ddzeta \eta^*_\zeta .
\label{e:fisheta}
\end{equation}
This is Poisson's equation, and it follows that $\cpi_\zeta(\util) =  \ddzeta \eta^*_\zeta $.  Moreover,  since $h^*_\zeta(\xz)=0$ for every $\zeta$,  we must have $H^*_\zeta(\xz)=0$ as well.  Since the solution to Poisson's equation with this normalization is unique, we conclude that  \eqref{e:Hrep} holds, and hence $H^*_\zeta = \preclH(\cP_\zeta)$ as claimed.

\section{Conclusions} 
\label{s:conc}

It is surprising that an MDP can be solved using an ODE under general conditions,  and fortunate that this ODE admits simple structure in the K-L cost framework that is a focus of the paper.   

It is likely that the ODE has special structure for other classes of MDPs, such as the  ``rational inattention'' framework of \cite{sim03,sim06,sharagmey13,sharagmey16}.  The computational efficiency of this approach  will depend in part on numerical properties of the ODE, such as its sensitivity for high dimensional models.   
\spm{dimension, or large state space -- this often confuses people}

 Finally, it is hoped that this approach will lead to new approaches to approximate dynamic programming or reinforcement learning.

\bigskip


\appendix
 
 \begin{center}
 \textbf{\large Appendices}
 \end{center}

\def\head#1{\smallbreak\noindent\textbf{#1}\quad }

\section{AROE and Duality}
 
Based on the linear programming (LP) approach to dynamic programming \cite{man60a,bor02a},  it can be shown that the AROE is the dual of the primal problem \eqref{e:ARobjectiveReward}.    The relative value function $h^*$ is the dual variable associated with the invariance constraint 
  $\pi=\pi P$ \cite{bor02a}.     To prove \Theorem{t:IPD} we require properties of the primal and dual.

The primal \eqref{e:ARobjectiveReward} is equivalently expressed, 
\begin{equation}
\eta_\zeta^* = 
\max_{\pi,R} \Bigl\{ \sum_x \pi(x)    \reward(x,R) \   :  \ \text{   \eqref{e:PQ0R} holds,  and  $\pi=\pi P$}\Bigr\}
\label{e:ARobjectiveRewardb}
\end{equation}
The AROE becomes,  
\begin{equation}
\max_R\Bigl\{  \reward(x,R) 
+\sum_{x'} P(x,x') h^*_\zeta(x') \Bigr\} = h^*_\zeta(x) + \eta^*_\zeta
\label{e:AROEb}
\end{equation}
It will be shown that \eqref{e:AROEb} can be interpreted as a dual of the convex program \eqref{e:ARobjectiveRewardb}.   We first characterize its optimizer, denoted $\cR_\zeta$.  This representation is based on the convex duality between K-L divergence and the log-moment generating function recalled in \Lemma{t:ConvexDual}.

Fix a  pmf $\mu_0$ on $\state$.  For any function $F\colon\state \to \Re$,  the log-moment generating function is denoted
\[
\Lambda(F) =\log\Bigl\{ \sum_x \mu_0(x) \exp(F(x)) \Bigr\}
\]
The mean of a function $F$ under an arbitrary pmf $\mu$ is denoted $\mu(F) = \sum_x \mu(x) F(x)$.  
The following  lemma can be regarded as a consequence of Kullback's inequality (see eqn (4.5) of \cite{kul54});  see also  Theorem 3.1.2  of  \cite{demzei98a}.
\spm{Let's do some more research on Kullback's inequality.  It seems some call this the DV lemma.  }

\begin{lemma}
\label{t:ConvexDual}
The log-moment generating function is the convex dual of relative entropy,
\[ 
\Lambda(F) = \max_\mu \{  \mu(F) - D(\mu\| \mu_0) \}
\]
where the maximum is over all pmfs, and is achieved uniquely with,
\[
 \mu^*(x) = \mu_0(x) \exp\{ F(x) - \Lambda(F) \},\qquad x\in\state.
\]
\qed
\end{lemma}

The following representation for $\cR_\zeta$ easily follows.  The fixed point equation appearing in \Prop{t:IPDfull} was previously stated in \eqref{e:NNfixedPt}.
\begin{proposition}
\label{t:IPDfull}
The control matrix maximizing the left hand side of 
\eqref{e:AROEb} 
is given by,
\begin{equation}
\cR_\zeta (x, x_u') = R_0(x,x_u') \exp \bigl( h^*_\zeta(x_u' \mid x) - \EFn{h^*_\zeta}(x) \bigr)\, .
\label{e:cR}
\end{equation}
Consequently,  the AROE is equivalent to the fixed point equation $
 \zeta \util   + \EFn{h^*_\zeta} = h^*_\zeta + \eta^*_\zeta $.
\end{proposition}
\begin{proof}  
Using \eqref{e:AROEcost}, the AROE becomes  $h^*_\zeta(x) + \eta^* =$
\[
\max_R\Bigl\{\reward(x,R) 
+\sum_{x_u', \, x_n'} R(x,x_u')Q_0(x,x_n')h^*_\zeta(x_u',x_n') \Bigr\} \, .
\]
Recalling the notation $ h^*_\zeta(x_u' \mid x) $ in \eqref{e:hmid},
we obtain $ h^*_\zeta(x) + \eta^* =$
 \begin{equation}
 \zeta \util(x)  +
\max_R\Bigl\{  \sum_{x_u'} R(x,x_u') h^*_\zeta(x_u' \mid x)  -   \sum_{x_u'} R(x,x_u')\log \Bigl(\frac{R(x,x'_u) }{R_0(x,x_u')} \Bigr)   \Bigr\} 
\label{e:AROEcb}
\end{equation}
  
For fixed $x$ denote $F(x_u') =  h^*_\zeta(x_u' \mid x) $,  $\mu_0(x_u') = R_0(x,x_u') $ and $\mu(x_u') = R(x,x_u') $,  $x_u'\in\stateu$.  The maximization variable in \eqref{e:AROEcb} is $\mu$,
and the maximization problem we must solve is,
\[
\max_\mu \{  \mu(F) - D(\mu\| \mu_0) \}
\]
The formula for the optimizer $ \mu^*$ in
\Lemma{t:ConvexDual} gives the expression for $\cR_\zeta$ in \eqref{e:cR}.

The fact that the optimal value is $\Lambda(F)$ implies the fixed point equation \eqref{e:NNfixedPt}.  
\end{proof}

It is established next that the AROE does indeed hold, by constructing a dual of \eqref{e:ARobjectiveRewardb} obtained though a relaxation of  the invariance constraint.  
A dual functional $\varphi^*_\zeta$ is defined for
 any function $h\colon\state\to\Re$ via
\[
\varphi^*_\zeta(h) =  \max_{\pi,R} \Bigl \{ \sum_x \pi(x) \bigl[ \reward(x,R)   + (P-I)h\, (x) \bigr]\Bigr\}
\]
where $(\pi,R)$ are now independent variables,  and $P$ is obtained from $R$ via   \eqref{e:PQ0R}.  
We have $\varphi^*_\zeta(h)\ge \eta_\zeta^*$ for any $h$, and there is no duality gap:  

\begin{proposition}
\label{t:AROEduality}
There exists $h^*_\zeta$ such that  $\varphi^*_\zeta(h^*_\zeta) = \eta_\zeta^*$.   
The pair $(h^*_\zeta,\eta_\zeta^*)$   is a solution to the AROE \eqref{e:AROE}.
\end{proposition}

\begin{proof} 
To show that there is no duality gap we apply    \Prop{t:IDPexists}, which establishes that the primal is a convex program,  and hence a sufficient condition is  Slater's condition \cite[Section~5.3.2]{boyvan04a}.  This condition holds because $\Pi_{\pi_0,R_0}$ is in the relative interior of the constraint-set for the primal.

Since there is no duality gap, it then follows that there exists a maximizer for $\varphi^*_\zeta$,  denoted $h^*_\zeta$, which satisfies $\eta^*_\zeta = \varphi^*_\zeta(h^*_\zeta) $.
To obtain the AROE, consider this representation:  $\eta^*_\zeta = $
\[
 \max_{\pi} \Bigl\{  \sum_x \pi(x)  \max_R \bigl[ \reward(x,R)   + Ph^*_\zeta\, (x) - h^*_\zeta(x)  \bigr] \Bigr\}
\]

The maximum over pmfs $\pi$ is the same as the maximum over $x$:
\[ 
\eta^*_\zeta = \max_x  \max_R \bigl\{  \reward(x,R)   + Ph^*_\zeta\, (x) - h^*_\zeta(x)  \bigr\}  
\]
To complete the proof we must remove the maximum over $x$.  For this, recall that $\pi_0$ and hence   $\cpi_\zeta$ have full support (they are strictly positive on all of $\state$).
 
\Prop{t:IPDfull} implies that the maximum over $R$ is uniquely given by  $\cR_\zeta$ in  \eqref{e:cR}, so that
\[ 
\eta^*_\zeta 
 = \max_x  \bigl\{  \reward(x,\cR_\zeta)   + \cP_\zeta h^*_\zeta\, (x) - h^*_\zeta(x)  \bigr\}  
\] 
Averaging over the optimizing  pmf $\cpi_\zeta$ gives, by invariance,
\[
\begin{aligned}
  \eta^*_\zeta 
&=  \sum_x \cpi_\zeta(x)    \reward(x,\cR_\zeta)  
\\
& =
  \sum_x \cpi_\zeta(x) \bigl\{  \reward(x,\cR_\zeta)   + \cP_\zeta h^*_\zeta\, (x) - h^*_\zeta(x)  \bigr\}  .
\end{aligned}
\]
Because  $ \cpi_\zeta(x) >0$ for every $x$,
it follows that the AROE \eqref{e:AROEb} holds: 
\[ 
\begin{aligned}
\eta^*_\zeta &= \reward(x,\cR_\zeta)   + \cP_\zeta h^*_\zeta\, (x) - h^*_\zeta(x) 
\\
&=   \max_R \bigl\{  \reward(x,R)   + Ph^*_\zeta\, (x) - h^*_\zeta(x)  \bigr\}    
\end{aligned}
 \]
\end{proof}

\section{Derivatives}

The proof of Part~(ii) of \Theorem{t:IPD} is obtained through a sequence of lemmas.    We first obtain an alternative representation for the fixed point equation
\eqref{e:NNfixedPt}.  Evaluating this equation at $\xz$,  and recalling that $h^*_\zeta(\xz)=0$ gives,
\begin{equation}
\eta^*_\zeta = \zeta \util(\xz)   + \EFn{h^*_\zeta} (\xz)
\label{e:etaForFP}
\end{equation}
Let $\One$ denote the function on $\state$ that is identically equal to $1$, and for any function $h$ and $\zeta\in\Re$ define a new function on $\state$ via
\begin{equation}
\clF(\zeta,  h)
\eqdef 
h - \zeta \util   - \EFn{h}  + [\zeta \util(\xz)   + \EFn{h } (\xz)] \One
\label{e:Ffix}
\end{equation} 
The fixed point equation becomes \eqref{e:NNfixedPt} becomes,
\begin{equation}
\clF(\zeta, h^*_\zeta) =0.  
\label{e:NNfixedPt-b}
\end{equation}

\spm{I have a pdf version of the textbook looste68.  See commented text for the link}

The proof of \Theorem{t:IPD} will require the Implicit Function Theorem.
The following version of this result is taken from \cite[Theorem 11.5]{looste68}.
\begin{proposition}[Implicit Function Theorem]
\label{t:looste68}
Suppose that $A\subset \Re^n$ and $B\subset \Re^m$ are open, and that $F\colon A\times B \to \Re^m$ is  continuously differentiable.  Suppose moreover that there exists $(x^0,y^0)\in 
A\times B$ for which the following hold:   $F(x^0,y^0) =0$,   and the matrix $\partial /\partial y F\, (x^0,y^0)$ has rank $m$.   

Then, there is a ball $O\subset A$ about $x^0$ and a continuously differentiable 
function $g\colon O\to B$ such the equation
  $F(x, y)=0$ is uniquely determined by $y=g(x)$,  for each $x\in O$.  
\qed
\end{proposition}

To apply \Prop{t:looste68}, we take $F=\clF$ and $(x,y)=(\zeta, h)$,  so that $n=1$ and $m=d$.
We apply the result to any $(\zeta_0,h^*_{\zeta_0})$ to establish that the mapping $\zeta\to h^*_\zeta$ is $C^1$.

For this
 we require a representation for the derivative of $\clF$ with respect to the variable $h$.  The derivative is represented as a $d\times d$ matrix, defined so that for any function $g\colon\state\to\state$,
\[
\clF(\zeta, h+\epsy g) \big|_x
=
\clF(\zeta,  h) \big|_x+ 
\epsy 
\sum_{x'\in\state}
 \frac{\partial}{\partial h} \clF\, (\zeta, h) \Bigr|_{x,x'} g(x') +o(\epsy)
\]
The following follows from \eqref{e:Ffix}
 and elementary calculus:
\begin{lemma}
\label{t:Fderh}
The function $\clF$ is continuously differentiable in $(\zeta, h)$. 
The partial derivative with respect to the second variable is,
\[
 \frac{\partial}{\partial h} \clF\, (\zeta, h)  = I - P_h + \One\otimes\nu , 
\]
in which $P_h$ is the transition matrix defined   in \eqref{e:Ph},
and $ \One\otimes \nu $ represents a $d\times d$ matrix with each row equal to $\nu$, and with
$
\nu(x) = P_h(\xz,x),\quad x\in\state$.
\qed
\end{lemma}

Invertibility of the derivative with respect to $h$ is obtained in the following:

\begin{lemma}
\label{t:Zh} 
The following inverse exists as a power series,
\[
Z_h =[I - P_h + \One\otimes\nu]^{-1}
= \sum_{n=0}^\infty  (P_h  - \One\otimes \nu  )^n
\]
in which $\nu$ is defined in \Lemma{t:Fderh}.
Moreover,   $  \nu Z_h$ is the unique invariant pmf for $P_h$. 
\end{lemma}

\begin{proof} 
It is easily established by induction that for each $n\ge 1$,
\[
 (P_h  - \One\otimes \nu  )^n = P_h^n - \One\otimes\nu_n,
\]
where $\nu_n= \nu P_h^{n-1}$. 
Recall that $P_h$ is irreducible and aperiodic since this is true for $P_0$.  Consequently, as $n\to\infty$ we have $\nu_n\to\pi_h$ and $ P_h^n\to \One\otimes\pi_h$, where $\pi_h$ is invariant for $P_h$.  The convergence is geometrically fast,  which establishes the desired inverse formula.

From the foregoing we have $ (P_h  - \One\otimes \nu  )^n\One = P_h^n\One  - \One = 0$ for $n\ge 1$, which implies that $Z_h\One=\One$.  
From this we obtain,
\[
Z_h P_h - \One\otimes\nu = Z_h(P_h  - \One\otimes \nu  ) = Z_h - I
\]
Multiplying each side by $\nu$ gives $ \nu Z_h P_h= \nu Z_h$, so that $\mu_h\eqdef \nu Z_h$ is invariant.  We also have $\mu_h(\state) =  \nu Z_h \One = \nu(\state)=1$,   where we have used again the identity $Z_h\One=\One$.  Hence $\mu_h = \pi_h$ as claimed.
\end{proof}

Since $h^*_\zeta$ is continuously differentiable in $\zeta$, it follows from
\eqref{e:NNfixedPt}
that the same is true for $\eta^*_\zeta$.  The following result provides a representation.   The formula for
$\ddzeta \eta_\zeta^*$ could be anticipated from
\Prop{t:etaConvex}.

\begin{lemma}
\label{t:LambdaFish} 
For each $\zeta$ we have,
\[
\ddzeta \EFn{h^*_\zeta}  = \cP_\zeta H^*_\zeta\, (x)
\]
where $H^*_\zeta =\ddzeta h^*_\zeta$,  and
\[
\ddzeta \eta_\zeta^* = \cpi_\zeta(\util)
\]
\end{lemma}
 
\begin{proof}
The first result holds by the definition of  $ \EFn{h}$  and $H^*_\zeta$.  To obtain the second identity,
we differentiate each side of \eqref{e:NNfixedPt} to obtain Poisson's equation \eqref{e:fisheta}.
On taking the mean of each side of  \eqref{e:NNfixedPt} with respect to $\cpi_\zeta$,  and using invariance $\cpi_\zeta \cP_\zeta =\cpi_\zeta$,  we obtain,
\[
  \cpi_\zeta( \util  ) +  \cpi_\zeta( H^*_\zeta ) = \cpi_\zeta(H^*_\zeta )+  \cpi_\zeta(\ddzeta \eta^*_\zeta ).
\]
\end{proof}
 
\begin{proof}[Proof of \Theorem{t:IPD}]
Part (i) is contained in \Prop{t:IPDfull}.

Part (ii):  Combining
\Lemma{t:Fderh}
and
\Lemma{t:Zh} 
we see that the conclusions of \Prop{t:looste68}
hold for each pair $(\zeta_0,h^*_{\zeta_0})$. This shows that $ h^*_\zeta$ is a continuously 
differentiable function of $\zeta$,  and hence $\cP_\zeta$ is also continuously differentiable.
To see that $\cpi_\zeta$ is
continuously 
differentiable, apply the representation in \Lemma{t:Zh}.
\end{proof}

\section{Optimal ODE solution}

We now prove \Theorem{t:IDPODE}.
 
The boundary condition is immediate from the assumptions:  $h_0^*$ is a constant, since $\cP_\zeta = P_0$.   Under the assumption that $h^*_\zeta(\xz)=0$ for each $\zeta$, it follows that $h_0^*(x) =0$ for each $x$.    It remains to show  that the relative value function solves the  ODE,
\[
  \ddzeta h_\zeta^* = \preclH(P_\zeta),
\]
with $\preclH$ defined in \eqref{e:fishP}.

On differentiating each side of  \eqref{e:NNfixedPt} we obtain,
\[
   \util(x)  +  \ddzeta \EFn{h^*_\zeta}(x)
		 = \ddzeta h_\zeta^*(x) + \ddzeta \eta_\zeta^*  
\]
Based on the definition \eqref{e:hmid}, 
\[
\EFn{h^*_\zeta}(x)=
\log\Bigl(   \sum_{x_u'} R_0(x,x_u') \exp(h^*_\zeta(x_u' \mid x))   \Bigr)  
\]
it follows that  $
\ddzeta \EFn{h^*_\zeta}(x)=$
\[
\begin{aligned} 
 \Bigl(   \sum_{x_u'} R_0(x,x_u') \exp(h^*_\zeta(x_u' \mid x) )   \Bigr)^{-1} 
 \sum_{x_u'} R_0(x,x_u') \exp(h^*_\zeta(x_u' \mid x)) \ddzeta h^*_\zeta(x_u' \mid x) 
\end{aligned}
\]
The equation simplifies as follows: \spm{this should be a general lemma}
\[
\begin{aligned}
\ddzeta \EFn{h^*_\zeta}(x) 
&=   \sum_{x_u'} \cR_\zeta (x,x_u') \ddzeta h^*_\zeta(x_u' \mid x) 
\\
&=   \sum_{x_u'} \cR_\zeta (x,x_u') \ddzeta \Bigl(\sum_{x_n'} Q_0(x,x_n') h^*_\zeta(x_u' , x_n') \Bigr) 
\\
  &= \sum_x \cP_\zeta(x,x') \ddzeta h^*_\zeta(x')
\end{aligned}
\]
Let $H=\ddzeta h_\zeta^*(x)$ and $\gamma =\ddzeta \eta_\zeta^*  $.  
 From the foregoing we obtain,
\[
\util(x) + \sum_{x'} \cP_\zeta(x,x') H(x')
=  H(x) + \gamma
\]
This is Poisson's equation, with $\gamma =\cpi_\zeta(\util)$.   
This shows that $h^*_\zeta$ is the solution to the ODE defined in \Theorem{t:IDPODE},  establishing (i) and (ii) of the theorem.

\bibliographystyle{siamplain}

\def\cprime{$'$}\def\cprime{$'$}

\end{document}